\documentclass[a4paper,12pt, reqno]{amsart}
\usepackage[ french, british]{babel}
\usepackage{a4wide}
\usepackage{amsthm}
\usepackage{amssymb}
\usepackage{longtable}

\newtheorem{thm}{Theorem}[section]
\newtheorem{cor}[thm]{Corollary}
\newtheorem{prop}[thm]{Proposition}
\newtheorem{lem}[thm]{Lemma}

\theoremstyle{definition}

\newtheorem{exmp}[thm]{Example}

\theoremstyle{remark}
\newtheorem{rem}[thm]{Remark}

\makeatletter
\let\c@equation\c@thm
\makeatother
\numberwithin{equation}{section}

\begin{document}
	
\author{S.Senthamarai Kannan and Pinakinath Saha}
	
\address{Chennai Mathematical Institute, Plot H1, SIPCOT IT Park, 
		Siruseri, Kelambakkam,  603103, India.}
	
\email{kannan@cmi.ac.in.}
	
\address{Chennai Mathematical Institute, Plot H1, SIPCOT IT Park, 
		Siruseri, Kelambakkam, 603103, India.}
\email{pinakinath@cmi.ac.in.}
\title{Parabolic subgroups and Automorphism groups of Schubert
varieties}

\begin{abstract}
		
		Let $G$ be a simple algebraic group of adjoint type over the field $\mathbb{C}$ 
		of complex numbers,  $B$ be a Borel subgroup of $G$ containing a maximal torus $T$ of $G.$	 Let $w$ be an element of the Weyl group $W$ and $X(w)$ be the Schubert variety in $G/B$
		corresponding to $w$. In this article we show that given any parabolic subgroup $P$ of $G$ containing $B$ properly, there is an element $w\in W$ such that $P$ is the  connected component, containing the identity element of the group of all algebraic automorphisms of $X(w).$ 
	
\end{abstract}	

\maketitle

\begin{otherlanguage}{french}

\begin{abstract}
Soit $G$ un groupe alg\'ebrique du type adjoint  sur le corps des nombres complexes $\mathbb{C}$ et $B$ un sous groupe de Borel de $G$ contenant un tore maximal $T.$ Soit  $w$ un \'el\'ement  du groupe de Weil $W$ et $X(w)$ un \'el\'ement de la vari\'et\'e de Schubert dans $G/B$ correspondant \`a $w.$ Dans cet article nous montrons que pour tout sous-groupe parabolique $P$ de $G$ contenant $B,$ il existe un \'el\'ement $w$ dans $W$ tel que   
$P$ est la composante connexe contenant l'\'el\'ement identit\'e du groupe des automorphismes alg\'ebrique de $X(w).$
\end{abstract}
\end{otherlanguage}
 
\section{Introduction}Recall that if $X$ is a projective variety over $\mathbb{C}$, the connected component, containing the identity element of the group of all algebraic automorphisms of $X$ is an algebraic group (see \cite[Theorem 3.7, p.17]{MO}). Let $G$ be a simple algebraic group of adjoint type over $\mathbb{C}.$ Let $T$  be a maximal torus  of $G,$  and let $R$ be the set of roots with respect to $T.$ Let $R^{+}\subset R$ be a set of positive roots. Let $B^{+}$ be the Borel subgroup of $G$ containing $T,$ corresponding to $R^{+}$. Let $B$ be the Borel subgroup of $G$ opposite to $B^+$ determined by $T$.
For $w \in W$, let $X(w):=\overline{BwB/B}$ denote the Schubert variety in $G/B$ corresponding to 
$w$.
Let $Aut^{0}(X(w))$ denote the connected component, containing the identity element of the group of all algebraic automorphisms of $X(w).$ 
Let $\alpha_{0}$ denote the highest root of $G$ with respect to $T$ and $B^{+}.$ For the left action of $G$ on $G/B$, let $P_{w}$ denote the stabiliser of $X(w)$ in $G.$
 If $G$ is simply laced and $X(w)$ is smooth, then we have $P_{w}=Aut^{0}(X(w))$ if and only if $w^{-1}(\alpha_{0})<0$ (see \cite[Theorem 4.2(2), p.772]{Ka}). Therefore it is a natural question to ask whether given any parabolic subgroup $P$ of $G$ containing $B$ properly, is there an element $w\in W$ such that $P=Aut^{0}(X(w))$ ? In this article we show that this question has an affirmative answer (see Theorem \ref {thm 2.1}). If $P=B,$ there is no such Schubert variety in $G/B.$ We prove some partial results for Schubert varities in partial flag varieties of type $A_{n}.$ If $P'$ is the maximal parabolic subgroup of $PSL(n+1, \mathbb{C})$ corresponding to the simple root $\alpha_{1}$ or $\alpha_{n},$  then $G/P'$ is the projective space $\mathbb{P}^n.$ The Schubert varieties in $\mathbb{P}^n$ are $\mathbb{P}^i$ ($0\le i \le n$). $\mathbb{P}^n$ is the only Schubert variety in $\mathbb{P}^{n}$ for which the action of $B$ is faithful. Further, we have $Aut^0(\mathbb{P}^n)=PSL(n+1, \mathbb{C})$ (see Corollary \ref {cor 6.4}). Therefore the answer to the above question is negative if we consider partial flag varieties. 

\section{Notation  and Result}
In this section, we set up some notation and preliminaries. We refer to \cite{BK}, \cite{Hum1}, \cite{Hum2}, \cite{Jan} for preliminaries in algebraic groups and Lie algebras.

Let $G$ be a simple algebraic group of adjoint type over $\mathbb{C}$ and $T$ be a maximal torus of
 $G$.  Let $W=N_{G}(T)/T$ denote the Weyl group of $G$ with respect to $T$ and we denote 
  the set of roots of $G$ with respect to $T$ by $R$. Let $B^{+}$ be a  Borel subgroup of $G$ 
  containing $T$. Let $B$ be the Borel subgroup of $G$ opposite to $B^{+}$ determined by $T$. 
  That is, $B=n_{0}B^{+}n_{0}^{-1}$, where $n_{0}$ is a representative in $N_{G}(T)$ of the longest     element $w_{0}$ of $W$. Let  $R^{+}\subset R$ be 
  the set of positive roots of $G$ with respect to the Borel subgroup $B^{+}$. Note that the set of 
  roots of $B$ is equal to the set $R^{-} :=-R^{+}$ of negative roots.

Let $S = \{\alpha_1,\ldots,\alpha_n\}$ denote the set of simple roots in
$R^{+}.$ For $\beta \in R^{+},$ we also use the notation $\beta > 0$.  
The simple reflection in  $W$ corresponding to $\alpha_i$ is denoted
by $s_{\alpha_i}$. Let $\mathfrak{g}$ be the Lie algebra of $G$. 
Let $\mathfrak{h}\subset \mathfrak{g}$ be the Lie algebra of $T$ and  $\mathfrak{b}\subset \mathfrak{g}$ be the Lie algebra of $B$. Let $X(T)$ denote the group of all characters of $T$. 
We have $X(T)\otimes \mathbb{R}=Hom_{\mathbb{R}}(\mathfrak{h}_{\mathbb{R}}, \mathbb{R})$, the dual of the real form of $\mathfrak{h}$. The positive definite 
$W$-invariant form on $Hom_{\mathbb{R}}(\mathfrak{h}_{\mathbb{R}}, \mathbb{R})$ 
induced by the Killing form of $\mathfrak{g}$ is denoted by $(~,~)$. 
We use the notation $\left< ~,~ \right>$ to
denote $\langle \mu, \alpha \rangle  = \frac{2(\mu,
\alpha)}{(\alpha,\alpha)}$,  for every  $\mu\in X(T)\otimes \mathbb{R}$ and $\alpha\in R$.  
We denote by $X(T)^+$ the set of dominant characters of 
$T$ with respect to $B^{+}$. Let $\rho$ denote the half sum of all 
positive roots of $G$ with respect to $T$ and $B^{+}.$
For any simple root $\alpha$, we denote the fundamental weight
corresponding to $\alpha$  by $\omega_{\alpha}.$  For $1\leq i \leq n,$ let $h(\alpha_{i})\in \mathfrak{h}$ be the fundamental coweight corresponding to $\alpha_{i}.$ That is ; $\alpha_{i}(h(\alpha_{j}))=\delta_{ij},$ where $\delta_{ij}$ is Kronecker delta.  

For $w \in W$, let $l(w)$ denote the length of $w$. We define the 
dot action of $W$ on $X(T)\otimes \mathbb{R}$ by $$w\cdot\lambda=w(\lambda+\rho)-\rho,  ~  where  ~ w\in W  ~ and ~ \lambda\in X(T)\otimes \mathbb{R}.$$ 
 We set $R^{+}(w):=\{\beta\in R^{+}:w(\beta)\in -R^{+}\}$.  For $w \in W$, 
let $X(w):=\overline{BwB/B}$ denote the Schubert variety in $G/B$ 
corresponding to $w$.

For a simple root $\alpha$, we denote by $P_{\alpha}$ the minimal parabolic subgroup of $G$  generated by $B$ and $n_{\alpha}$, where $n_{\alpha}$ is a representative of $s_{\alpha}$ in $N_{G}(T)$ and we denote by $P_{\hat{\alpha}}$ the maximal parabolic subgroup of $G$ generated by $B$ and $\{n_{\beta}: \beta \in S\setminus \{\alpha\}\},$ where $n_{\beta}$ is a representative of $s_{\beta}$ in $N_{G}(T).$ For a subset $J$ of $S,$ we denote by $W_{J}$ the subgroup of $W$ generated by $\{s_{\alpha}:\alpha \in J\}$. Let $W^{J}:=\{w\in W: w(\alpha)\in R^{+}~ for ~ all ~ \alpha \in J\}.$  For each $w\in W_{J},$ choose a representative element $n_{w}\in N_{G}(T).$ Let $N_{J}:=\{n_{w}: w\in W_{J}\}.$  Let $P_{J}:=BN_{J}B.$ 

Our main result in this article is the following :

\begin{thm}\label{thm 2.1}
Let $G$ be a simple algebraic group of adjoint type over $\mathbb{C}$ and $P$ be a parabolic subgroup of $G$ containing $B$ properly. Then there is an element $w\in W$ such that $P=Aut^{0}(X(w))$.
\end{thm}

Let $G=PSL(n+1, \mathbb{C}).$ For $1\le r\le n$ and $w\in W^{S\setminus\{\alpha_{r}\}},$ we denote the Schubert variety corresponding to $w$ in the Grassmannian $G/P_{\hat{\alpha}_{r}},$ by $X_{P_{\hat {\alpha}_{r}}}(w).$ 
 
\begin{prop}
	Let $w=(s_{a_{1}}\cdots s_{1})(s_{a_{2}}\cdots s_{2})\cdots (s_{a_{r}}\cdots s_{r})\in W(r).$ Let $J'(w):=\{i\in \{1,2,\dots,r-1\}: a_{i+1}-a_{i}\ge 2\}$, $J''(w)=\{1+a_{i}: i\in J'(w)\}$ and $J(w)=\{\alpha_{j}: j\in \{1, \ldots ,n\}\setminus J''(w)\}.$
	Then we have $P_{J(w)}=Aut^0(X_{P_{\hat {\alpha}_{r}}}(w))$.
	
	\end{prop} 

For more precise statement see Proposition \ref{prop 4.2}.

\section{proof of theorem 2.1 except in three cases}
In this section we prove Theorem \ref {thm 2.1} in all cases except in three cases. The three cases left will be treated by Proposition \ref{BCG}.


 
 
\begin{proof}
Let $P$ be a parabolic subgroup of $G$ containing $B$ properly. If $P=G,$ then we take $w=w_{0}$, the longest element $w_0$ of $W.$ In this case, we have the following:
	
\begin{center}
		$Aut^0(X(w_0))=Aut^0(G/B)= G$ (see \cite[Theorem 2, p.75]{Akh}).
\end{center}
	
Now we assume that $P$ is any proper parabolic subgroup of $G$ such that $B \subsetneq P \subsetneq G.$ Since  $B \subsetneq P \subsetneq G,$ there is a subset $\emptyset \neq I \subsetneq S$  such that $P=P_{I}.$ Consider  $J=S \setminus I$.  Hence, there exist unique elements $ w_{0}^{J} \in  W^{J}$ and ${w_{0}},_{J}\in W_{J}$ such that $w_{0}=w_{0}^J\cdot {w_{0}},_{J}$.
Consider the natural left action of $G$ on $G/B$. Take $w=(w_{0}^{J})^{-1}$. Then $P$ is the stabiliser of $X(w)$, since $R^{+}(w^{-1})\cap S=I.$ The natural action of $P$ on $X(w)$ induces a homomorphism,
\begin{center}
		$\phi_{w}:P\longrightarrow Aut^{0}(X(w))$
\end{center} 
of algebraic groups.
	
We note that $\phi_{w}:P\longrightarrow Aut^{0}(X(w))$ is injective, since $w^{-1}(\alpha_{0})<0$ (see \cite[Theorem 4.2(2), p.772]{Ka}).

	Let $J':=-w_{0}(J),$ and $P':=P_{J'}.$ Consider the natural morphism $\pi: G/B \longrightarrow G/P'.$ We denote the restriction of $\pi$ to $X(w)$ also by $\pi.$ Then 
	$\pi: X(w)\longrightarrow G/P'$ is a birational morphism. Therefore by \cite[Theorem 3.3.4(a), p.96]{BK} and 
	\cite[Lemma 3.3.3(b), p.95]{BK}  we have,
	
\begin{center}
		$\pi_{*}(\mathcal{O}_{X(w)})=\mathcal{O}_{G/P'}.$
\end{center}
	
Thus from \cite[Corollary 2.2., p.45]{Bri}, $\pi$ induces a homomorphism of algebraic groups
	$$\pi_{*}:Aut^{0}(X(w))\longrightarrow Aut^{0}(G/P').$$
Since $\pi$ is birational, $\pi_{*}:Aut^{0}(X(w))\longrightarrow Aut^{0}(G/P')$ is injective.
	
If $G$ is of type $B_{n}, C_{n}$ or $G_{2},$ then $w_{0}=-id$ (see \cite[p.216, p.217, p.233]{B}). If $G$ is of type $B_{n}$ and $P=P_{\alpha_{n}},$ then $I=\{\alpha_{n}\}.$ Therefore $J'=-w_{0}(J)=J =S \setminus \{\alpha_{n}\}$ and $P'=P_{\hat{\alpha_{n}}}.$ Thus $(G, P')$ is one of the three types as in the statement of \cite[Theorem 2, p.75]{Akh}.
If $G$ is of the type $C_{n}$ and $P=P_{\alpha_{1}},$ then $(G, P')=(G, P_{\hat{\alpha_{1}}})$ is one of the three types as in the statement of \cite[Theorem 2, p.75]{Akh}. 
If $G$ is of type $G_{2}$ and $P=P_{\alpha_{1}}$ then $(G, P')=(G, P_{\hat{\alpha_{1}}})=(G, P_{\alpha_{2}})$ is one of the three types as in the statement of \cite[Theorem 2, p.75]{Akh}.
Similarly, we can see that if $(G, P')$ is one of the three types as in \cite[Theorem 2, p.75]{Akh}, then $(G, P)$ is one of the three types as in the statement of Proposition \ref{BCG}.
	
 Case 1: $G$ is not of type $B_{n}, C_{n}$ and $G_{2}.$ Then for any parabolic subgroup $P$ of $G,$ $(G, P)$ is not one of the three types as in Proposition \ref{BCG}. Therefore $(G, P')$ is not one of the three exceptional types as in the statement of \cite[Theorem 2, p.75]{Akh}.
	
 Case 2: $G=B_{n}$ or $C_{n}$ or $G_{2}$ and $(G, P)$ is not one of the three types as in the statement of Proposition \ref{BCG}. In these cases $w_{0}=-id$ and  $J'=-w_{0}(J)=J=S\setminus I.$ Therefore $(G, P')$ is not one of the three exceptional types as in the statement of \cite[Theorem 2, p.75]{Akh}. Thus $(G, P)$ is not one of the three types as in the statement of Proposition \ref{BCG} if and only if $(G, P')$ is not one of the three exceptional types as in the statement of \cite[Theorem 2, p.75]{Akh}.	
Hence, we have $Aut^{0}(G/P')=G.$ Therefore $Aut^0(X(w))$ is a parabolic subgroup of $G$ containing $P.$ Since $P$ is the stabiliser of $X(w),$ we have $P=Aut^0(X(w)).$ Now, the proof follows from the proofs of Case 1 and Case
2.
\end{proof}

\section{preliminaries for three left cases}

Let $V$ be a rational $B$-module. Let $\phi:B\longrightarrow GL(V)$ be the corresponding homomorphism of algebraic groups. The total space of the vector bundle  $\mathcal{L}(V)$  on $G/B$ is defind by the set of equivalence classes 
$\mathcal{L}(V)= G \times_{B} V$ corresponding to the following equivalence relation on $G\times V$:
\begin{center}
	$(g,v)\sim (gb,\phi(b^{-1})\cdot v)$ for $g\in G, b\in B, v\in V.$ 
\end{center}
We denote the restriction of $\mathcal{L}(V)$ to $X(w)$ also by $\mathcal{L}(V)$. We denote the cohomology modules $H^i(X(w), \mathcal{L}(V))$ by $H^i(w, V)$ ($i\in \mathbb{Z}_{\ge 0}$). If $V=\mathbb{C}_{\lambda}$ is one dimensional representation $\lambda: B\longrightarrow \mathbb{C}^{\times}$ of $B,$ then we denote $H^i(w, V)$ by $H^i(w, \lambda).$

Let $L_{\alpha}$ denote the Levi subgroup of $P_{\alpha}$
containing $T$. Note that $L_{\alpha}$ is the product of $T$ and the homomorphic image 
$G_{\alpha}$ of $SL(2, \mathbb{C})$ via a homomorphism $\psi:SL(2, \mathbb{C})\longrightarrow L_{\alpha}$ ( see  [7, II, 1.3] ). We denote the intersection of $L_{\alpha}$ and $B$ by $B_{\alpha}$.  
We note that the morphism $L_{\alpha}/B_{\alpha}\hookrightarrow P_{\alpha}/B$  induced by the inclusion $L_{\alpha}\hookrightarrow P_{\alpha}$ is an isomorphism. Therefore,  to compute the cohomology modules $H^{i}(P_{\alpha}/B, \mathcal{L}(V))$ ($0\leq i \leq 1$) for any $B$-module 
$V,$ we treat $V$ as a $B_{\alpha}$-module  and we compute 
$H^{i}(L_{\alpha}/B_{\alpha}, \mathcal{L}(V))$. 

We use the following lemma to compute cohomology groups. The following lemma is due to Demazure (see \cite[p.1]{Dem}). He used this lemma to prove  Borel-Weil-Bott's theorem.

\begin{lem}\label{lem 2.1}
	Let $w=\tau s_{\alpha},$ $l(w)=l(\tau) + 1,$ and $\lambda$ be a character of $B.$ Then we have 
	\begin{enumerate}
		
		\item [(1)] If $\langle \lambda , \alpha \rangle\ge0,$ then $H^{j}(w, \lambda)=H^{j}(\tau, H^0(s_{\alpha}, \lambda))$ for all $j\ge 0.$
		
		\item[(2)] If $\langle \lambda , \alpha \rangle\ge 0,$ then $H^j(w, \lambda)=H^{j+1}(w, s_{\alpha}\cdot \lambda)$ for all $j\ge0.$ 
		
		\item[(3)] If $\langle \lambda, \alpha\rangle\le -2 ,$ then $H^{j+1}(w, \lambda)=H^j(w, s_{\alpha}\cdot \lambda)$ for all $j\ge 0.$
		
		\item[(4)] If $\langle \lambda , \alpha \rangle=-1,$ then $H^j(w, \lambda)$ vanishes for every $j\ge0.$
	\end{enumerate}
\end{lem}

Let $\pi:\hat{G}\longrightarrow G$ be the simply connected covering of $G$.  
Let  $\hat{L_{\alpha}}$  (respectively, $\hat{B_{\alpha}}$) be the inverse image 
of $L_{\alpha}$ (respectively, of $B_{\alpha}$) in $\hat{G}$. Note that $\hat{L_{\alpha}}/\hat{B_{\alpha}}$ is isomorphic to $L_{\alpha}/B_{\alpha}$. We make use of this isomorphism to use the same notation for the vector bundle on $L_{\alpha}/B_{\alpha}$ associated to a $\hat{B_{\alpha}}$-module. Let $V$ be an irreducible  $\hat{L_{\alpha}}$-module and $\lambda$ be a character of $\hat{B_{\alpha}}$. 

Then, we have

\begin{lem}\label{lem 3.1}
	\begin{enumerate}
		
		\item If $\langle \lambda , \alpha \rangle \geq 0$, then, the $\hat{L_{\alpha}}$-module
		$H^{0}(L_{\alpha}/B_{\alpha} , V\otimes \mathbb{C}_{\lambda})$ 
		is isomorphic to the tensor product of  $ ~ V$ and 
		$H^{0}(L_{\alpha}/B_{\alpha} , \mathbb{C}_{\lambda})$. Further, we have  
		$H^{j}(L_{\alpha}/B_{\alpha} , V\otimes \mathbb{C}_{\lambda}) =0$ 
		for every $j\geq 1$.
		\item If $\langle \lambda , \alpha \rangle  \leq -2$, then, we have  
		$H^{0}(L_{\alpha}/B_{\alpha} , V \otimes \mathbb{C}_{\lambda})=0.$ 
		Further, the $\hat{L_{\alpha}}$-module  $H^{1}(L_{\alpha}/B_{\alpha} , V \otimes \mathbb{C}_{\lambda})$ is isomorphic to the tensor product of  $V$ and $H^{0}(L_{\alpha}/B_{\alpha} , 
		\mathbb{C}_{s_{\alpha}\cdot\lambda})$. 
		\item If $\langle \lambda , \alpha \rangle = -1$, then 
		$H^{j}( L_{\alpha}/B_{\alpha} , V \otimes \mathbb{C}_{\lambda}) =0$ 
		for every $j\geq 0$.
	\end{enumerate}
\end{lem}

\begin{proof}
	
	By \cite[I, Proposition 4.8, p.53]{Jan} and \cite[I, Proposition 5.12, p.77]{Jan} for $j\ge0$,
	we have the following isomorphism as 
	$\hat{L_{\alpha}}$-modules:
	$$H^j(L_{\alpha}/B_{\alpha}, V \otimes \mathbb C_{\lambda})\simeq V \otimes
	H^j(L_{\alpha}/B_{\alpha}, \mathbb C_{\lambda}).$$ 
	Now, the proof of the lemma follows from Lemma \ref{lem 2.1} by taking $w=s_{\alpha}$ 
	and the fact that $L_{\alpha}/B_{\alpha} \simeq P_{\alpha}/B$.
\end{proof}

We now state the following Lemma on indecomposable 
$\hat{B_{\alpha}}$ (respectively, $B_{\alpha}$) modules which will be used in computing 
the cohomology modules (see \cite [Corollary 9.1, p.30]{BKS}).

\begin{lem}\label{lem 3.2}
	\begin{enumerate}
		\item
		Any finite dimensional indecomposable $\hat{B_{\alpha}}$-module $V$ is isomorphic to 
		$V^{\prime}\otimes \mathbb{C}_{\lambda}$ for some irreducible representation
		$V^{\prime}$ of $\hat{L_{\alpha}}$, and some character $\lambda$ of $\hat{B_{\alpha}}$.
		
		\item
		Any finite dimensional indecomposable $B_{\alpha}$-module $V$ is isomorphic to 
		$V^{\prime}\otimes \mathbb{C}_{\lambda}$ for some irreducible representation
		$V^{\prime}$ of $\hat{L_{\alpha}}$, and some character $\lambda$ of $\hat{B_{\alpha}}$.
	\end{enumerate}
\end{lem}
\begin{proof}
	Proof of part (1) follows from \cite [Corollary 9.1, p.30]{BKS}.
	
	Proof of part (2) follows from the fact that every $B_{\alpha}$-module can be viewed  as a 
	$\hat{B_{\alpha}}$-module via the natural homomorphism. 	
\end{proof}

\section{Proof of theorem \ref {thm 2.1} in three left cases}
To complete the proof of Theorem \ref {thm 2.1}, it is sufficient to prove the following proposition.  By $(G, P)$ we mean $G$ is a simple algebraic group of adjoint type over $\mathbb{C}$ and $P$ is a parabolic subgroup of $G$ containing $B.$ 

\begin{prop}\label{BCG}
	Let $(G, P)$ be one of the following types.
	\begin{itemize}
		\item[(1)] $G$ is of type $B_{n}$ and $P=P_{\alpha_{n}}$ is the minimal parabolic subgroup of $G$ corresponding to 
		$\alpha_{n}.$
		\item[(2)] $G$ is of type $C_{n}$ and $P=P_{\alpha_{1}}$ is the minimal parabolic subgroup of $G$ corresponding to 
		$\alpha_{1}.$
		\item[(3)] $G$ is of type $G_{2}$ and $P=P_{\alpha_{1}}$ is the minimal parabolic subgroup of $G$ corresponding to 
		$\alpha_{1}.$
	\end{itemize} Then, there exists an element $w\in W$ such that $P=Aut^{0}(X(w))$.
\end{prop}
\begin{proof}
	
	Let $T_{X(w)}$ be the tangent sheaf of $X(w).$ Let $T_{G/B}$ be the restriction of the tangent bundle to $X(w).$ Then $T_{X(w)}$ is a subsheaf of $T_{G/B}$ on $X(w).$ By \cite[Lemma 3.4, p.13]{MO} we have
	$Lie(Aut^0(X(w))=H^{0}(X(w), T_{X(w)})\subset H^0(X(w), T_{G/B})=H^{0}(w, \mathfrak{g}/\mathfrak{b}).$
	
	As in the strategy of proof in Section 3, it is sufficient to prove that for all the three types $(G, P)$ as above, there is an element $w\in W$ such that 
	
	\begin{itemize} 
		\item[(i)] $P$ is the stabiliser of $X(w)$ in $G.$ 
		\item [(ii)] $w^{-1}(\alpha_{0}) < 0.$ 
		\item [(iii)] $H^{0}(w, \mathfrak{g}/\mathfrak{b})=\mathfrak{g}.$  
	\end{itemize}
	
	For instance, let $\phi_{w}:P\longrightarrow Aut^0(X(w))$ be the natural homomorphism induced by the action of $P$ on $X(w).$
	
	Since $w^{-1}(\alpha_{0})<0,$ $\phi_{w}: P\longrightarrow Aut^0(X(w))$ is injective. Since $H^0(w, \mathfrak{g/b})=\mathfrak{g},$ we have $H^0(X(w), T_{X(w)})\subseteq \mathfrak{g}.$ Therefore $Aut^0(X(w))$ is a closed subgroup of $G$ containing $P.$ Since $P$ is the stabilizer of $X(w)$ in $G,$ we have $P=Aut^0(X(w)).$
	
	We first make a note about statement (ii) and statement (iii). Let  $w\in W$  be such that  $w^{-1}(\alpha_{0}) < 0.$  To prove that $H^{0}(w, \mathfrak{g}/\mathfrak{b})=\mathfrak{g},$ it is sufficient     
	to prove that for any negative root $\beta,$ the dimension of the weight space $H^{0}(w, \mathfrak{g}/\mathfrak{b})_{\beta}$ is one.
	
	Proof of this note: 
	
	The restriction of the natural map $\mathfrak{g}\longrightarrow \mathfrak{g}/\mathfrak{b}$ to 
	$\bigoplus\limits_{\alpha\in R^{+}}^{}\mathfrak{g}_{\alpha}$ is an isomorphism of $T$-modules and hence, we have $\mathfrak{g}/\mathfrak{b}=\bigoplus\limits_{\alpha\in R^{+}}\mathbb{C}_{\alpha}.$ Since $s_{i}$ permutes all positive roots other than $\alpha_{i}$ for every $1\leq i \leq n,$ 
	every indecomposable $B_{\alpha_{i}}$-summand $V$ of $\mathfrak{g}/\mathfrak{b}$
	with highest weight, a positive root different from $\alpha_{i}$ is indeed an $\hat{L}_{\alpha_{i}}$-module 
	and hence for every $\alpha\in R^{+}\setminus S,$ the dimension of the weight space $H^{0}(s_{i}, \mathfrak{g}/\mathfrak{b})_{\alpha}$ is one. Using this argument and by induction on length of $w$ 
	we see that  the dimension of the weight space $H^{0}(w, \mathfrak{g}/\mathfrak{b})_{\alpha}$ is one 
	for every $\alpha\in R^{+}\setminus S.$ Further, since $(\mathfrak{g}/\mathfrak{b})_{\alpha}$ is one dimensional for every simple root $\alpha,$ each fundamental coweight $h(\alpha_{i})$ ($1\leq i \leq n$) appears exactly once.  Hence, it is sufficient to prove that  for any negative root $\beta$
	the dimension of the weight space $H^{0}(w, \mathfrak{g}/\mathfrak{b})_{\beta}$ is one.

We prove the existence of an element $w\in W $ satisfying the first two conditions and 
that the dimension of the weight space $H^{0}(w, \mathfrak{g}/\mathfrak{b})_{\beta}$ is one 
for any negative root $\beta$ in all the three cases separately.

\vspace{.2cm}
Case 1: Assume that $G$ is of type $B_{n}$ and $P=P_{n}.$ For every $1\leq r \leq n-1,$ let $v_{r}=s_{n}s_{n-1}\cdots s_{r}.$ Take $w=v_{1}v_{2}\cdots v_{n-1}.$  It is easy to see that $P_{n}$ is the stabiliser of $X(w).$ 

In this case $\alpha_{0}=\omega_{2}.$  So, we have $v_{1}^{-1}(\alpha_{0})=\alpha_{2}+2(\sum\limits_{i=3}^{n}\alpha_{i}).$ This is the highest root of type $B_{n-1}$ corresponding to the root system whose set of simple roots is $S\setminus \{\alpha_{1}\}.$   By induction on rank of $G$, we have $w^{-1}(\alpha_{0})=(v_{2}\cdots v_{n-1})^{-1} (\alpha_{2}+2(\sum\limits_{i=3}^{n}\alpha_{i})) < 0.$ 

Now, if $v\in W$ is of minimal length such that the dimension of $H^{0}(v, \mathfrak{g}/\mathfrak{b})_{\beta}$ is at least two for some negative root $\beta,$ then $\beta=-(\sum\limits_{j=i}^{n}\alpha_{j})$ for some $1\leq i \leq n-1$.

Justification of the above statement:
Clearly for any such $v$, $l(v)>1.$ Choose $\gamma \in S$ such that
$l(s_{\gamma}v)=l(v)-1.$ Let $u=s_{\gamma}v.$ 

Then we have dim$H^0(s_{\gamma}, H^0(u, \mathfrak{g}/\mathfrak{b}))_{\beta}\ge 2.$

If $\langle \beta , \gamma \rangle=1$, then there exists an indecomposable $B_{\gamma}$-summand $V$ of $H^0(u, \mathfrak{g/b})$ such that $H^0(u, V)_{\beta}\neq0.$
In this case, either $V=\mathbb{C}_{\beta}\oplus \mathbb{C}_{\beta -\gamma}$ or $V=\mathbb{C}_{\beta}.$

So we have   dim$H^0(s_{\gamma}, H^0(u, \mathfrak{g}/\mathfrak{b}))_{\beta}=1.$

If  $\langle \beta , \gamma \rangle=-1$, we have, either $V=\mathbb{C}_{\beta}\oplus \mathbb{C}_{\beta +\gamma}$ or $V=\mathbb{C}_{\beta + \gamma}.$

So we have
dim$H^0(s_{\gamma}, H^0(u, \mathfrak{g}/\mathfrak{b}))_{\beta}=1.$

If $\langle \beta , \gamma \rangle=2$ then there exists a unique indecomposable $B_{\gamma}$-summand $V$ of $H^0(u, \mathfrak{g/b})$ with highest weght $\beta.$

Therefore    dim$H^0(s_{\gamma}, H^0(u, \mathfrak{g}/\mathfrak{b}))_{\beta}=1.$

If  $\langle \beta , \gamma \rangle=-2$ then there exists a unique indecomposable $B_{\gamma}$-summand of $H^0(u, \mathfrak{g/b})$ with highest weight $\beta + 2\gamma.$ Therefore    dim$H^0(s_{\gamma}, H^0(u, \mathfrak{g}/\mathfrak{b}))_{\beta}=1.$

Following the case by case analysis as above, we conclude that $\langle \beta , \gamma  \rangle=0$ and there is a unique indecomposable $B_{\gamma}$-summand $V$ of $H^0(u, \mathfrak{g/b})$ such that $V=\mathbb{C}_{\beta +\gamma}\oplus \mathbb{C}_{\beta}.$
In particular, we have $\beta + \gamma\in R^-.$  Since $G$ is of type $B_{n}$, we have $\gamma=\alpha_{n}$ and $\beta=-(\sum\limits_{j=i}^{n}\alpha_{j})$ for some $1\le i\le n-1.$  

By induction on the rank of $G,$ we may assume that  
$H^{0}(v_{2}v_{3}\cdots v_{n-1}, \mathfrak{g}/\mathfrak{b})_{-(\sum\limits_{j=i}^{n}\alpha_{j})}$ is one dimensional for every $2\leq i \leq n-1.$  Also $H^{0}(v_{2}v_{3}\cdots v_{n-1}, \mathfrak{g}/\mathfrak{b})_{-(\sum\limits_{j=1}^{n}\alpha_{j})}=0.$ 

Since $\langle \sum\limits_{j=i}^{n}\alpha_{j} , \alpha_{1} \rangle = 0$ for every $3\leq i \leq n-1,$  the restriction of the evaluation map

$H^{0}(s_{1}v_{2}v_{3}\cdots v_{n-1}, \mathfrak{g}/\mathfrak{b})_{-(\sum\limits_{j=i}^{n}\alpha_{j})} \longrightarrow H^{0}(v_{2}v_{3}\cdots v_{n-1}, \mathfrak{g}/\mathfrak{b})_{-(\sum\limits_{j=i}^{n}\alpha_{j})}$

is an isomorphism for every $3\leq i \leq n-1$ (see Lemma \ref {lem 2.1} and Lemma \ref{lem 3.1}).

Since $\langle -( \sum\limits_{j=2}^{n}\alpha_{j}), \alpha_{1} \rangle = 1,$ we have 

$H^{0}(s_{1}v_{2}v_{3}\cdots v_{n-1}, \mathfrak{g}/\mathfrak{b})_{-(\sum\limits_{j=i}^{n}\alpha_{j})}=H^{0}(s_{1}, H^{0}(v_{2}v_{3}\cdots v_{n-1}, \mathfrak{g}/\mathfrak{b}))_{-(\sum\limits_{j=i}^{n}\alpha_{j})}$

is one dimensional  for every $i=1,2$ (see Lemma \ref {lem 2.1} and Lemma \ref{lem 3.1}).

Now, it is easy to see that for every  $2\leq r \leq n$ the evaluation map

$H^{0}(s_{r}s_{r-1}\cdots s_{2}, H^{0}(s_{1}v_{2}v_{3}\cdots v_{n-1}, \mathfrak{g}/\mathfrak{b}))_{-(\sum\limits_{j=i}^{n}\alpha_{j})}\longrightarrow  H^{0}(s_{1}v_{2}v_{3}\cdots v_{n-1}, \mathfrak{g}/\mathfrak{b})_{-(\sum\limits_{j=i}^{n}\alpha_{j})}$

is an isomorphism for every $1\leq i \leq n$ by induction on $r$ and using Lemma \ref {lem 2.1}, Lemma \ref{lem 3.1}. Thus, the space $H^{0}(w, \mathfrak{g}/\mathfrak{b})_{\alpha}$ is one dimensional for every negative root $\alpha.$

\vspace{.2cm}
Case 2: Assume that $G$ is of type $C_{n}$ ( $n\geq 3$ ) and $P=P_{1}.$  Take $w=s_{1}s_{2}\cdots s_{n}.$ In this case we have $\alpha_{0}=2\omega_{1},$ and $w^{-1}(\alpha_{0})=-\alpha_{n}.$
Further, the stabiliser of $X(w)$ in $G$ is $P_{1}.$ 

First note that 

$H^{0}(s_{n}, \mathfrak{g}/\mathfrak{b})=(\bigoplus\limits_{\alpha\in R^{+}}\mathbb{C}_{\alpha})
\oplus \mathbb{C}_{h(\alpha_{n})}\oplus \mathbb{C}_{-\alpha_{n}}$ (see Lemma \ref {lem 2.1} and Lemma \ref{lem 3.1}).  

Further, we have 

$H^{0}(s_{n-1}s_{n}, \mathfrak{g}/\mathfrak{b})=H^{0}(s_{n-1} , H^{0}(s_{n}, \mathfrak{g}/\mathfrak{b}))
=(\bigoplus\limits_{\alpha\in R^{+}}\mathbb{C}_{\alpha})\oplus \mathbb{C}_{h(\alpha_{n})}\oplus \mathbb{C}_{-\alpha_{n}}\oplus \mathbb{C}_{h(\alpha_{n-1})}$ \\ $\oplus \mathbb{C}_{-\alpha_{n-1}}
\oplus \mathbb{C}_{-(\alpha_{n-1}+\alpha_{n})}\oplus \mathbb{C}_{-(2\alpha_{n-1}+\alpha_{n})}$ (see Lemma \ref {lem 2.1} and Lemma \ref{lem 3.1}).

By using Lemma \ref {lem 2.1}, Lemma \ref{lem 3.1} and the descending induction on $1\leq r \leq n-1,$  we see that

$$H^{0}(s_{r}\cdots s_{n-1} s_{n}, \mathfrak{g}/\mathfrak{b})
=(\bigoplus\limits_{\alpha\in R^{+}}\mathbb{C}_{\alpha})\oplus (\bigoplus_{i=r}^{n}\mathbb{C}_{h(\alpha_{i})})\oplus \mathbb{C}_{-\mu},$$ where $\mu$ runs over all positive roots in $\sum_{i=r}^{n}\mathbb{Z}_{\geq 0}\alpha_{i}.$  Thus, we have $H^{0}(w, \mathfrak{g}/\mathfrak{b})=\mathfrak{g}.$

\vspace{.2cm}
Case 3:  Assume that $G$ is of type $G_{2}$ and $P=P_{1}.$ Take $w=s_{1}s_{2}s_{1}s_{2}.$ 
Here, we follow the convention in \cite{Hum1}.  In this case, we have $\alpha_{0}=3\alpha_{1}+2\alpha_{2}.$ Further, $w^{-1}(\alpha_{0})=-\alpha_{2}.$   

First note that $H^{0}(s_{2}, \mathfrak{g}/\mathfrak{b})=(\bigoplus\limits_{\alpha\in R^{+}}\mathbb{C}_{\alpha})
\oplus \mathbb{C}_{h(\alpha_{2})}\oplus \mathbb{C}_{-\alpha_{2}}$ (see Lemma \ref {lem 2.1} and Lemma \ref{lem 3.1}). 

$H^{0}(s_{1},  H^{0}( s_{2}, \mathfrak{g}/\mathfrak{b}))
=(\bigoplus\limits_{\alpha\in R^{+}}\mathbb{C}_{\alpha})
\oplus \mathbb{C}_{h(\alpha_{2})}\oplus \mathbb{C}_{-\alpha_{2}} \oplus \mathbb{C}_{h(\alpha_{1})}\oplus \mathbb{C}_{-\alpha_{1}}\oplus (\bigoplus\limits_{i=1}^{3}\mathbb{C}_{-(\alpha_{2} + i\alpha_{1})})$ (see Lemma \ref {lem 2.1} and Lemma \ref{lem 3.1}). 

Therefore we have

$H^{0}(s_{1}s_{2}, \mathfrak{g}/\mathfrak{b})=(\bigoplus\limits_{\alpha\in R^{+}}\mathbb{C}_{\alpha})
\oplus \mathbb{C}_{h(\alpha_{2})}\oplus \mathbb{C}_{-\alpha_{2}} \oplus \mathbb{C}_{h(\alpha_{1})}\oplus \mathbb{C}_{-\alpha_{1}}\oplus (\bigoplus\limits_{i=1}^{3}\mathbb{C}_{-(\alpha_{2} + i\alpha_{1})}).$

$H^{0}(s_{2} , H^{0}(s_{1}s_{2}, \mathfrak{g}/\mathfrak{b}))=(\bigoplus\limits_{\alpha\in R^{+}}\mathbb{C}_{\alpha})
\oplus \mathbb{C}_{h(\alpha_{2})}\oplus \mathbb{C}_{-\alpha_{2}} \oplus \mathbb{C}_{h(\alpha_{1})}\oplus \mathbb{C}_{-\alpha_{1}}\oplus (\bigoplus\limits_{i=1}^{3}\mathbb{C}_{-(\alpha_{2} + i\alpha_{1})})\oplus \mathbb{C}_{-(3\alpha_{1}+2\alpha_{2})}=\mathfrak{g}$ (see Lemma \ref {lem 2.1} and Lemma \ref{lem 3.1}). 

Therefore we have

$H^{0}(s_{2}s_{1}s_{2}, \mathfrak{g}/\mathfrak{b})=
(\bigoplus\limits_{\alpha\in R^{+}}\mathbb{C}_{\alpha})
\oplus \mathbb{C}_{h(\alpha_{2})}\oplus \mathbb{C}_{-\alpha_{2}} \oplus \mathbb{C}_{h(\alpha_{1})}\oplus \mathbb{C}_{-\alpha_{1}}\oplus (\bigoplus\limits_{i=1}^{3}\mathbb{C}_{-(\alpha_{2} + i\alpha_{1})})\oplus \mathbb{C}_{-(3\alpha_{1}+2\alpha_{2})}.$ 

Thus, we have $H^{0}(w, \mathfrak{g}/\mathfrak{b})=H^{0}(s_{1}, \mathfrak{g})=\mathfrak{g}.$

\end{proof}

\begin{exmp}
Let $G=PSL(3,\mathbb{C}).$ In this case, $B$ is the set of invertible lower triangular matrices, $P_{\alpha_{1}}=Aut^0(X(s_{1}s_{2}))$ and $X(s_{1}s_{2})$ is smooth.
\end{exmp}
\begin{rem}
In Theorem \ref{thm 2.1}, for a given parabolic subgroup $P$ of $G$ containing $B$ properly, the Schubert variety $X(w)$  for which $P=Aut^{0}(X(w))$ is not necessarily smooth. For example, take $G=PSL(4, \mathbb{C}),$ and $P_{\alpha_{2}}=Aut^{0}(X(s_{2}s_{1}s_{3}s_{2})).$ Note that $X(s_{2}s_{1}s_{3}s_{2})$ is not smooth (see \cite[Theorem 2.2, p.48]{LS}). 
\end{rem}

\section{automorphism groups of schubert varieties in partial flag varieties of type $A_{n}$}

In this section we discuss about parabolic subgroups of $G=PSL(n+1, \mathbb{C})$ and connected component, containing identity element of the group of all algebraic automorphisms of Schubert varieties in the Grassmannian $G/P_{\hat{\alpha}_{r}},$ where $1\le r\le n$ and $P_{\hat{\alpha}_{r}}=P_{S\setminus\{\alpha_{r}\}}.$  

\begin{lem}\label{lem 4.1}
Let $G=PSL(n+1, \mathbb{C}).$ Let $1\le r\le n$ and  $w\in W^{S\setminus \{\alpha_{r}\}}.$ Then
$w^{-1}(\alpha_{0})<0$ if and only if there exists an increasing sequence $1\le a_{1}<a_{2}<\dots <a_{r}=n$ of positive integers such that $w=(s_{a_{1}}\cdots s_{1})(s_{a_{2}}\cdots s_{2})\cdots (s_{a_{r}}\cdots s_{r}).$
\end{lem} 
\begin{proof}
Note that $\alpha_{0}=\alpha_{1}+\alpha_{2}+\dots+\alpha_{n}.$
Let  $w\in W^{S\setminus \{\alpha_{r}\}}$ be such that $w\neq id.$ Then there exists an integer $1\le i\le r$ and  an increasing sequence of positive integers $i\le a_{i}<a_{i+1}<\dots <a_{r}\le n$ such that $w=(s_{a_{i}}\cdots s_{i})(s_{a_{i+1}}\cdots s_{i+1})\cdots (s_{a_{r}}\cdots s_{r}).$
Now, it is easy to see that $w^{-1}(\alpha_{0})<0$ if and only if  $i=1$ and $a_{r}=n.$
\end{proof}

Let $W(r)=\{w\in W^{S\setminus \{\alpha_{r}\}}: w=(s_{a_{1}}\cdots s_{1})(s_{a_{2}}\cdots s_{2})\cdots (s_{a_{r}}\cdots s_{r}),$ where $1\le a_{1}<a_{2}<\dots <a_{r}=n\}.$ For $w\in W^{S\setminus \{\alpha_{r}\}},$ we denote the Schubert variety in the Grassmannian $G/P_{\hat{\alpha}_{r}}$ corresponding to $w$ by $X_{P_{\hat {\alpha}_{r}}}(w).$ 
\begin{prop}\label{prop 4.2}
Let $w=(s_{a_{1}}\cdots s_{1})(s_{a_{2}}\cdots s_{2})\cdots (s_{a_{r}}\cdots s_{r})\in W(r).$ Let $J'(w):=\{i\in \{1,2,\dots,r-1\}: a_{i+1}-a_{i}\ge 2\}$, $J''(w)=\{1+a_{i}: i\in J'(w)\}$ and $J(w)=\{\alpha_{j}: j\in \{1, \ldots ,n\}\setminus J''(w)\}.$
Then we have $P_{J(w)}=Aut^0(X_{P_{\hat {\alpha}_{r}}}(w)).$
\end{prop} 
\begin{proof}Let $P_{w}$ be the stabiliser of $X_{P_{\hat {\alpha}_{r}}}(w)$ in $G.$
First we show that $P_{w}=P_{J(w)}.$ 
If $a_{i+1}-a_{i}\ge 2$ for some $1\le i\le r-1$ then  $s_{1+a_{i}}w>w,$ and  $s_{1+a_{i}}w\in  W^{S\setminus \{\alpha_{r}\}}.$
Hence $s_{1+a_{i}}$ is not in the Weyl group of $P_{w}.$ 
Therefore $P_{w}$ is a subgroup of  $P_{J(w)}.$
Let $R(P_{\hat {\alpha}_{r}})=R\cap(\sum\limits_{\alpha \in S\setminus \{\alpha_{r}\}}\mathbb{Z}\alpha).$
Further, it is easy to see that for $\alpha \in J(w)$  either we have $w^{-1}(\alpha)<0$ or  $w^{-1}(\alpha)\in R(P_{\hat {\alpha}_{r}}).$ Therefore $P_{J(w)}\subseteq P_{w}$.

Let $\psi_{w}: P_{J(w)}\longrightarrow Aut^0(X_{P_{\hat{\alpha}_{r}}}(w))$ be the natural homomorphism induced by action of $P_{J(w)}$ on $X_{P_{\hat{\alpha}_{r}}}(w).$ 

Since $w\in W(r),$  $w^{-1}(\alpha_{0})<0$ (see Lemma \ref{lem 4.1}). Therefore $\psi_{w}: P_{J(w)}\longrightarrow Aut^0(X_{P_{\hat{\alpha}_{r}}}(w))$ is injective.

Let $\mathfrak{p}_{{\hat{\alpha}}_{r}}$ be the Lie algebra of $P_{\hat {\alpha}_{r}}.$ Since $G$ is simply laced, the restriction map \\
$H^0(w_{0,r}, \mathfrak{g/p_{{\hat{\alpha}}_{r}}}) \longrightarrow H^0(w, \mathfrak{g/p_{{\hat{\alpha}}_{r}}})$ is surjective, where $w_{0,r}\in W^{S\setminus\{\alpha_{r}\}}$ is the minimal representative of $w_{0}$ (see \cite[Lemma 3.5(3), p.770]{Ka}).

Further, since $w^{-1}(\alpha_{0})<0,$ $H^0(w_{0,r} , \mathfrak{g/p_{{\hat{\alpha}}_{r}}})=\mathfrak{g}\longrightarrow H^0(w, \mathfrak{g/p_{{\hat{\alpha}}_{r}}})$ is an isomorphism.

Therefore we have $H^0(X_{P_{\hat{\alpha}_{r}}}(w), T_{X_{P_{\hat{\alpha}_{r}}}(w)})\subseteq \mathfrak{g}.$ Hence $Aut^0(X_{P_{\hat{\alpha}_{r}}}(w))$ is  a closed subgroup of $G$ containing $P_{J(w)}.$ Thus we have  $P_{J(w)}=Aut^0(X_{P_{\hat{\alpha}_{r}}}(w)).$  
\end{proof}

\begin{cor}
	Let $B\subsetneq P$ be a parabolic subgroup of $G$ and $w\in W^{S\setminus \{\alpha_{r}\}}$  such that $P=Aut^0(X_{P_{\hat{\alpha}_{r}}}(w)).$ Then we have  $P=P_{J(w)}.$
\end{cor}
\begin{cor}\label{cor 6.4}
\begin{enumerate}
		
\item [(1)]

If $P\neq G,$ then there is no element $w\in W^{S\setminus \{\alpha_{1}\}}$
such that $P=Aut^0(X_{P_{\hat{\alpha}_{1}}}(w)).$ 

\item[(2)] 

If $P\neq G,$ then there is no element $w\in W^{S\setminus \{\alpha_{n}\}}$
such that $P=Aut^0(X_{P_{\hat{\alpha}_{n}}}(w)).$ 

\end{enumerate}
\begin{proof}
Proof of (1): The Schubert varieties in $G/P_{\hat{\alpha_{1}}}$ are projective space $\mathbb{P}^i$ ($0\le i \le n$). Therefore the automorphism groups of these Schubert varieties are $PSL(i+1, \mathbb{C})$ $(0\le i\le n)$. Further, the map $\phi_{w}$ is injective for only one $w.$ 

Proof of (2): Proof of (2) is similar to that of (1).

\end{proof}

\end{cor}

{\bf Acknowledgements} We are grateful to the Infosys Foundation for the partial financial support. We are grateful for the referee for suggesting the reference of the book D. N. Akhiezer, Lie group actions in complex analysis, which helps to improve the exposition of this article.

\end{document}